\documentclass[11pt]{35}
\usepackage{amsthm, amsmath, amssymb, amsbsy, amsfonts, latexsym, euscript}
\usepackage{calrsfs} 
\usepackage{graphicx}

\DeclareMathAlphabet{\varmathbb}{U}{pxsyb}{m}{n}

\def\leq{\leqslant}

\def\geq{\geqslant}
\def\phi{\varphi}

\def\tilde{\widetilde}

\def\kappa{\varkappa}

\newcommand{\D}{\mathrm{d}\kern0.2pt}%

\newcommand{\E}{\mathrm{e}\kern0.2pt} 
\newcommand{\ii}{\kern0.05em\mathrm{i}\kern0.05em}

\newcommand{\RR}{\mathbb{R}}%

\newtheorem{theorem}{\bf \indent Theorem}[section]

\newtheorem{corollary}{\bf \indent Corollary}[section]

\theoremstyle{remark}

\newtheorem{conjecture}{\bf \indent Conjecture}[section]

\numberwithin{equation}{section}

\begin{document}

\noindent {\Large \bf A characterisation of harmonic functions by quadrature \\[3pt]
identities of annular domains and related results}

\vskip5mm

{\bf Nikolay Kuznetsov}

\vskip-2pt {\small Laboratory for Mathematical Modelling of Wave Phenomena}
\vskip-4pt {\small Institute for Problems in Mechanical Engineering, RAS} \vskip-4pt
{\small V.O., Bol'shoy pr. 61, St Petersburg 199178, Russia \vskip-4pt {\small
nikolay.g.kuznetsov@gmail.com}

\vskip7mm

\parbox{146mm} {\noindent A new characterization of harmonic functions is obtained.
It is based on quadrature identities involving mean values over annular domains and
over concentric spheres lying within these domains or on their boundaries. The
analogous result with a logarithmic weight in the volume means is conjectured. The
similar characterization is derived for real-valued solutions of the modified
Helmholtz equation (panharmonic functions), whereas for solutions of the Helmholtz
equation (metaharmonic functions) the analogous result is just outlined. Weighted
mean value identities are also obtained for solutions of these equations in
two-dimensions. The dependence of coefficients in these identities on a logarithmic
weight is analysed and characterizations of these functions are conjectured.}

\vskip4mm

{\centering \section{Introduction} }

\noindent 

Studies of mean value properties of harmonic functions date back to the Gauss
theorem of the arithmetic mean over spheres (see \cite{G}, Article~20):

\begin{theorem}
Let $u \in C^2 (D)$ be harmonic in a domain $D \subset \RR^m$, $m \geq 2$. Then for
every $x \in D$ the equality $M^\circ (S_r (x), u) = u (x)$ holds for each
admissible sphere~$S_r (x)$.
\end{theorem}

Here and below the following notation and terminology are used. The open ball of
radius $r$ centred at $x$ is $B_r (x) = \{ y : |y-x| < r \}$; it is called
admissible with respect to a domain $D$ provided $\overline{B_r (x)} \subset D$, and
$S_r (x) = \partial B_r (x)$ is the corresponding admissible sphere. If $u \in C^0
(D)$, then its spherical mean value over $S_r (x) \subset D$ is
\begin{equation}
M^\circ (S_r (x), u) = \frac{1}{|S_r|} \int_{S_r (x)} u (y) \, \D S_y =
\frac{1}{\omega_m} \int_{S_1 (0)} u (x +r y) \, \D S_y \, ; \label{sm}
\end{equation}
here $|S_r| = \omega_m r^{m-1}$ and $\omega_m = 2 \, \pi^{m/2} / \Gamma (m/2)$ is
the total area of the unit sphere (as usual $\Gamma$ stands for the Gamma function),
and $\D S$ is the surface area measure.

Integrating $M^\circ (S_r (x), u)$ with respect to $r$, one obtains the following
mean value property over balls.

\begin{theorem}
Let $u \in C^2 (D)$ be harmonic in a domain $D \subset \RR^m$, $m \geq 2$. Then for
every $x \in D$ the equality $M^\bullet (B_r (x), u) = u (x)$ holds for each
admissible ball~$B_r (x)$.
\end{theorem}

The volume mean of a locally Lebesgue integrable $u$ is defined similarly to
\eqref{sm}:
\begin{equation}
M^\bullet (D', u) = \frac{1}{|D'|} \int_{D'} u (y) \, \D y \, ; \label{bm}
\end{equation}
here $D'$ is a measurable subset of $D$ and $|D'|$ is its volume.

Various improvements of Theorems 1.1 and 1.2 and related results are reviewed in
\cite{NV}. The recent survey \cite{Ku1} complements \cite{NV} with old and new
properties of mean values of harmonic functions not covered by Netuka and Vesel\'y.
However, both papers practically do not touch upon quadrature identities (see
\cite{Gu}, p.~211, for their definition) other than those involving spheres and
balls, whereas several authors investigated such identities, in particular, for
annular domains in two and higher dimensions.

Namely, a combination of \eqref{bm}, in which $D'$ is the following annular domain
\[ A (r_1, r_2) = \{ x \in \RR^m \,: r_1 < |x| < r_2 \} , \quad 0 \leq r_1 < r_2
< \infty ,
\]
and the mean value over sphere appears in the quadrature identity
\begin{equation} 
M^\bullet (A (r_1, r_2), u) = M^\circ (S_{r_*} (0), u) \, . \label{ag}
\end{equation}
It was obtained by Armitage and Goldstein \cite{AG} for any function $u$ harmonic in
$A (r_1, r_2)$ and integrable over this domain; they considered the case $m \geq 3$
when
\begin{equation}
r_* = \left( \frac{2}{m} \frac{r_2^m - r_1^m}{r_2^2 - r_1^2} \right)^{1/(m-2)} .
\label{r*}
\end{equation}
If $m = 2$, then identity \eqref{ag} is valid with
\begin{equation}
r_* = \exp \left( \frac{r_2^2 \log r_2 - r_1^2 \log r_1}{r_2^2 - r_1^2} - \frac{1}{2}
\right) \label{r*2}
\end{equation}
as was established by Sakai \cite{S}. Both expressions \eqref{r*} and \eqref{r*2}
are such that $r_* \in (r_1, r_2)$.

Moreover, it occurs that identity \eqref{ag} with $A (r_1, r_2)$ changed to $D$
characterises annuli in the class of domains containing $S_{r_*} (0)$ inside and
bounded by more than one surface (curve when $m=2$); see \cite{GS} for the precise
formulation and proof.

These results naturally lead to the following observation. Let $u$ be harmonic in a
domain containing $A (r_1, r_2)$ inside and having the connected complement, then
Theorem 1.2 implies that
\[ |B_{r_2} (0)| \, u (0) = \int_{B_{r_2} (0)} u (y) \, \D y \quad \mbox{and} \quad
|B_{r_1} (0)| \, u (0) = \int_{B_{r_1} (0)} u (y) \, \D y \, .
\]
By subtraction we obtain that $M^\bullet (A (r_1, r_2), u) = u (0)$. On the other
hand, Theorem~1.1 implies that $M^\circ (S_{r} (0), u) = u (0)$ for every $r \in (0,
r_2)$. Hence the quadrature identity
\begin{equation}
M^\bullet (A (r_1, r_2), u) = M^\circ (S_{r} (0), u) \label{A=S0}
\end{equation}
is valid for every $r \in [r_1, r_2]$. This provides the motivation to consider the
following characterisation of harmonic functions.

\begin{theorem}
Let $D$ be a bounded domain in $\RR^m$, $m \geq 2$. If $u \in C^0 (D)$ satisfies the
identity
\begin{equation}
M^\bullet (A_x (r_1, r_2), u) = M^\circ (S_{r} (x), u) \label{A=S}
\end{equation}
for every $x \in D$ and every $A_x (r_1, r_2) = \{ y \in \RR^m \,: r_1 < |x - y| < r_2
\}$ with $r_2 > r_1 > 0$ such that $B_{r_2} (x)$ is admissible, whereas $r$ either
attains each value in $(r_1, r_2)$ or is equal to $r_i$, $i = 1$ or $2$, then $u$ is
harmonic in $D$.
\end{theorem}

This is an analogue of Koebe's theorem \cite{K} (see also the survey \cite{NV}, p.
362), which is the converse of Theorem~1.1. The difference between these theorems
is that annuli and spheres surrounding every $x \in D$ are used in \eqref{A=S},
whereas just spheres centred at $x$ are involved in Koebe's theorem.

In conclusion of this section, we mention that there are numerous papers aimed at
relaxing the assumptions imposed in \cite{K}; see, for example, \cite{HN} and
references cited therein. Of course, Theorem 1.3 also admits improvements; the
obvious one ensuing from the proof below is as follows: it is sufficient to require
that $r$ attains only a sequence of values tending to $r_2$ within $(r_1, r_2)$
instead of attaining each value in this interval. Further improvements are left to
the reader.

{\centering \section{Proof of Theorem 1.3} }

\noindent Prior to proving our theorem, we recall a result on which our proof
relies; it dates back to 1943.

\begin{theorem}[Beckenbach and Reade, \cite{BR}]
Let $D \subset \RR^m$, $m \geq 2$, be a bounded domain. If $u \in C^0 (D)$ satisfies
the identity
\begin{equation}
M^\bullet (B_r (x), u) = M^\circ (S_{r} (x), u) \label{B=S}
\end{equation}
for every $x \in D$ and all $r$ such that $B_r (x)$ is admissible, then $u$ is
harmonic in $D$.
\end{theorem}

\begin{proof}
Let $\rho$ be a small positive number. If $r \in (0, \rho)$, then $M^\bullet (B_r
(x), u)$ is defined for $x$ belonging to an open subset of $D$ that depends on the
smallness of~$\rho$. Moreover, $M^\bullet (B_r (x), u)$ is differentiable with
respect to $r$. Then we have
\[ \partial_r M^\bullet (B_r (x), u) = m r^{-1} [ M^\circ (S_r (x), u) - M^\bullet
(B_r (x), u) ] = 0 \quad \mbox{for} \ r \in (0, \rho) ,
\]
where the last equality is a consequence of \eqref{B=S}. Hence $M^\bullet (B_r (x),
u)$ does not depend on $r \in (0, \rho)$. On the other hand, shrinking $B_r (x)$ to
its centre by letting $r \to 0$ and taking into account that $u \in C^0 (D)$, one
obtains that 
\[ M^\bullet (B_r (x), u) \to u (x) \quad \mbox{as} \ r \to 0 ,
\]
when $x$ belongs to an arbitrary closed subset of $D$. Hence for every $x \in D$ we
have
\[ u (x) = M^\bullet (B_r (x), u) = M^\circ (S_{r} (x), u)
\]
for all $r \in (0, \rho)$ with some $\rho$, whose smallness depends on $x$. Then
Koebe's theorem yields that $u$ is harmonic in $D$.
\end{proof}

This proof reproduces the original one, but in the $m$-dimensional setting, whereas
Beckenbach and Reade restricted themselves to $m = 2$. Subsequently, a
generalization of Theorem~2.1 to subharmonic functions was obtained in \cite{FM};
namely, {\it the inequality
\begin{equation}
M^\bullet (B_{r} (x), u) \leq M^\circ (S_{r} (x), u) \label{B<S}
\end{equation}
holds for every $x \in D$ and all $r$ such that $B_r (x)$ is admissible, if and only
if $u$ is subharmonic in $D$.} It is clear that this assertion implies Theorem~2.1.
Moreover, \eqref{B<S} implies that {\it for every $x \in D$ and all $r \in [r_1,
r_2]$ the inequality
\begin{equation*}
M^\bullet (A_x (r_1, r_2), u) \leq M^\circ (S_{r} (x), u)
\end{equation*}
holds provided $B_{r_2} (x)$ is admissible and $u$ is subharmonic.} It is
interesting whether the converse of this assertion is true.

\begin{proof}[Proof of Theorem 1.3]
Let us begin with the case when $r = r_1$, that is, identity \eqref{A=S} takes the
form
\begin{equation*}
M^\bullet (A_x (r_1, r_2), u) = M^\circ (S_{r_1} (x), u) \, ,
\end{equation*}
and holds for every $x \in D$ and every $A_x (r_1, r_2)$. Letting $r_1 \to 0$ (this
is possible because the identity is valid for all annuli and in view of continuity
of $u$), we obtain
\[ M^\bullet (B_{r_2} (x), u) = u (x) \, ,
\]
which holds for every $x \in D$ and every admissible $B_{r_2} (x)$. Hence $u$ is
harmonic in~$D$ by Koebe's theorem.

Assuming that $r = r_2$ in identity \eqref{A=S}, we see that it turns
into~\eqref{B=S} with $r = r_2$ in the limit $r_1 \to 0$, and the latter holds for
every $x \in D$ and every admissible $B_{r_2} (x)$. Hence $u$ is harmonic in $D$ by
Theorem~2.1.

If $r$ attains each value in $(r_1, r_2)$, then letting $r \to r_2$ in the original
identity \eqref{A=S}, we arrive at the same identity with $r = r_2$. Then the
considerations of the previous paragraph are applicable, thus yielding that $u$ is
harmonic in $D$ by Theorem~2.1.
\end{proof}

{\centering \section{A conjecture involving a weighted mean} }

In the recent note \cite{Ku2}, two remarkable properties of the weight function
$\log \frac{r}{|x - y|}$ were established. The first of them is the following
analytic characterisation of balls in $\RR^m$.

\begin{theorem}
Let $D \subset \RR^m$ be a bounded domain. If the identity
\begin{equation}
\frac{1}{m} = \frac{1}{|D|} \int_D \log \frac{r}{|x - y|} \, \D y \label{MW''}
\end{equation}
holds with $x \in D$ and $r > 0$ such that $|D| \geq |B_r|$, then $D = B_r (x)$.
\end{theorem}

The second one is the weighted mean value property of harmonic functions.

\begin{theorem}
Let $u \in C^2 (D)$ be harmonic in a domain $D \subset \RR^m$, $m \geq 2$. Then
\begin{equation}
u (x) = \frac{m}{|B_r|} \int_{B_r (x)} u (y) \log \frac{r}{|x-y|} \, \D y 
\label{har}
\end{equation}
for every $x \in D$ and each admissible ball $B_r (x)$.
\end{theorem}

Remarkably, the weighted mean \eqref{har} does not depend on $r$ (this is not true
for the mean values of harmonic functions involving other weights considered in
\cite{Ku2}). Therefore, combining the last assertion and considerations preceding
Theorem 1.3, one arrives at the following.

\begin{conjecture}
Let $D$ be a bounded domain in $\RR^m$, $m \geq 2$. If $u \in C^0
(D)$ satisfies the identity
\begin{equation}
\frac{m}{|A (r_1, r_2)|} \int_{A_x (r_1, r_2)} u (y) \log \frac{r}{|x-y|} \, \D y =
M^\circ (S_{r} (x), u) \label{con}
\end{equation}
for every $x \in D$ and every $A_x (r_1, r_2)$ with $r_2 > r_1 > 0$ such that
$B_{r_2} (x)$ is admissible, whereas $r$ either attains each value in $(r_1, r_2)$
or is equal to $r_i$, $i = 1$ or $2$, then $u$ is harmonic in $D$.
\end{conjecture}

Identity \eqref{con} is similar to \eqref{A=S}, but the presence of a logarithmic
weight in the annular mean makes this assertion substantially more difficult than
Theorem~1.3. Indeed, the converse of Theorem~3.2 is not proved yet, and so it is
worth considering this converse first. Then the conjecture will immediately follow
in the same way as Theorem~1.3.

{\centering \section{Characterisation of pan- and metharmonic functions} }

\noindent We begin this section with considering real-valued $C^2$-solutions of the
$m$-dimensional modified Helmholtz equation:
\begin{equation}
\nabla^2 u - \mu^2 u = 0 , \quad \mu \in \RR \setminus \{0\} ;
\label{Hh}
\end{equation}
$\nabla = (\partial_1, \dots , \partial_m)$ is the gradient operator, $\partial_i =
\partial / \partial x_i$. Unfortunately, it is not commonly known that these
solutions are called panharmonic functions (or $\mu$-panharmonic functions) by
analogy with solutions of the Laplace equation; this convenient abbreviation coined
by Duffin \cite{D} will be used in what follows.

A new proof of the following analogue of Theorem 1.1 for panharmonic functions was
obtained in~\cite{Ku3}.

\begin{theorem}
Let $u$ be panharmonic in a domain $D \subset {\RR}^m$, $m \geq 2$. Then for every
$x \in D$ the identity
\begin{equation}
M^\circ (S_r (x), u) =  a^\circ (\mu r) \, u (x) , \quad a^\circ (\mu r) = \Gamma
\left( \frac{m}{2} \right) \frac{I_{(m-2)/2} (\mu r)}{(\mu r / 2)^{(m-2)/2}} \, ,
\label{MM}
\end{equation}
holds for each admissible sphere $S_r (x);$ $I_\nu$ denotes the modified Bessel
function of order $\nu$.
\end{theorem}

The function $a^\circ$ increases monotonically on $[0, \infty)$ from $a^\circ (0) =
1$ to infinity. This is guaranteed by properties of $I_\nu$; see \cite{Wa}, pp. 79,
80. For $m=3$ the mean value formula \eqref{MM} has particularly simple form.
Indeed,
\[ a^\circ (\lambda r) = \sinh (\mu r) / (\mu r) \, ,
\]
because $I_{1/2} (z) = \sqrt{2 / (\pi z)} \sinh z$. Integrating \eqref{MM} with
respect to $r$, one arrives at the following.

\begin{corollary}[\cite{Ku3}]
Let $D$ be a domain in ${\RR}^m$, $m \geq 2$. If $u$ is panharmonic in $D$, then
\begin{equation}
M^\bullet (B_r (x), u) =  a^\bullet (\mu r) \, u (x) , \quad a^\bullet (\mu r) =
\Gamma \left( \frac{m}{2} + 1 \right) \frac{I_{m/2} (\mu r)}{(\mu r / 2)^{m/2}} \, ,
\label{MM'}
\end{equation}
and
\begin{equation} 
a^\circ (\mu r) M^\bullet (u, x, r) = a^\bullet (\mu r) M^\circ (u, x, r) 
\label{MMtil}
\end{equation}
for every admissible ball $B_r (x)$.
\end{corollary}

Here, identity \eqref{MM'} is analogous to that in Theorem 1.2, whereas
\eqref{MMtil} generalizes \eqref{B=S}. The following converse of this corollary was
established in \cite{Ku4}; see Theorems~3.1 and 3.2 of this paper.

\begin{theorem}
Let $D$ be a bounded domain in $\RR^m$. If $u \in C^0 (D)$ satisfies identity
\eqref{MM'} with $\mu > 0$ for every $x \in D$ and for all $r \in (0, r (x))$, where
$B_{r (x)} (x)$ is admissible, then $u$ is $\mu$-panharmonic in $D$.

Moreover, if identity \eqref{MMtil} is fulfilled for $u$ in the same way as above,
then $u$ is $\mu$-panharmonic in $D$.
\end{theorem}

Here, the first (second) assertion is similar to Koebe's theorem (Theorem~2.1).

The same considerations as those preceding Theorem~1.3, but using \eqref{MM} and
\eqref{MM'} instead of Theorems~1.1 and 1.2, respectively, lead to the following
generalisation of the quadrature identity \eqref{A=S0}
\begin{equation*}
M^\bullet (A (r_1, r_2), u) = \frac{a^\bullet (\mu r_2) |B_{r_2} (0)| - a^\bullet
(\mu r_1) |B_{r_1} (0)|}{a^\circ (\mu r) |A (r_1, r_2)|} M^\circ (S_{r} (0), u) ,
\end{equation*}
which is valid for every $r \in [r_1, r_2]$ provided $B_{r_2} (0)$ belongs to the
domain, where $u$ is panharmonic. This suggests the following analogue of
Theorem~1.3 for panharmonic functions.

\begin{theorem}
Let $D \subset \RR^m$, $m \geq 2$, be a bounded domain. If $u \in C^0 (D)$
satisfies the identity
\begin{equation}
M^\bullet (A_x (r_1, r_2), u) = \frac{a^\bullet (\mu r_2) |B_{r_2} (x)| - a^\bullet
(\mu r_1) |B_{r_1} (x)|}{a^\circ (\mu r) |A_x (r_1, r_2)|} M^\circ (S_{r} (x), u) \, ,
\ \ where \ \mu > 0 ,
\label{A=Span}
\end{equation}
for every $x \in D$ and every $A_x (r_1, r_2)$ with $r_2 > r_1 > 0$ such that
$B_{r_2} (x)$ is admissible, whereas $r$ either attains each value in $(r_1, r_2)$
or is equal to $r_i$, $i = 1$ or $2$, then $u$ is $\mu$-panharmonic in $D$.
\end{theorem}

\begin{proof}
Following the proof of Theorem 1.3, we begin with the case when $r = r_1$. In view
of continuity and the behaviour of $a^\circ$, letting $r_1 \to 0$ in
\eqref{A=Span}, we obtain
\[ M^\bullet (B_{r_2} (x), u) = a^\bullet (\mu r_2) \, u (x) \, , 
\]
which holds for every $x \in D$ and every admissible $B_{r_2} (x)$. Hence $u$ is
$\mu$-panharmonic in $D$ by the first assertion of Theorem~4.2.

Assuming that $r = r_2$ in identity \eqref{A=Span}, we see that in the limit $r_1
\to 0$ it turns into \eqref{MMtil} with $r = r_2$, which holds for every $x \in D$
and every admissible $B_{r_2} (x)$. Then $u$ is $\mu$-panharmonic in $D$ by the
second assertion of Theorem~4.2.

If $r$ attains each value in $(r_1, r_2)$, then letting $r \to r_2$ in the original
identity \eqref{A=Span}, we arrive at the same identity with $r = r_2$. Then the
considerations of the previous paragraph are applicable, thus yielding that $u$ is
$\mu$-panharmonic in $D$ by the second assertion of Theorem~4.2.
\end{proof}

In conclusion of this section, we  turn to real-valued metaharmonic functions; the
term is just an abbreviation for `$C^2$-solution to the Helmholtz equation':
\begin{equation}
\nabla^2 u + \lambda^2 u = 0 , \quad \lambda \in \RR \setminus \{0\} \, .
\label{Hh}
\end{equation}
I.~N. Vekua introduced this term in 1943, in his still widely cited article, which
was also published as Appendix~2 to the monograph \cite{Ve2}.

It occurs that the mean value properties of metaharmonic functions have the same
form as those of panharmonic ones, but the functions $I_{(m-2)/2}$ and $I_{m/2}$
used in Theorem~4.1 and Corollary~4.1 for defining the coefficients $a^\circ$ and
$a^\bullet$, respectively, must be changed to the corresponding Bessel functions
which decay at infinity oscillating about the zero level. In view of this similarity,
we just formulate the analogue of Theorem~4.3 and explain why an extra assumption is
imposed in it.

\begin{theorem}
Let $D \subset \RR^m$, $m \geq 2$, be a bounded domain. If $u \in C^0 (D)$
satisfies the identity
\begin{equation}
M^\bullet (A_x (r_1, r_2), u) = \frac{a^\bullet (\lambda r_2) |B_{r_2} (x)| -
a^\bullet (\lambda r_1) |B_{r_1} (x)|}{a^\circ (\lambda r) |A_x (r_1, r_2)|} M^\circ
(S_{r} (x), u) \, , \ \ where \ \lambda > 0 ,
\label{A=Span}
\end{equation}
for every $x \in D$ and every $A_x (r_1, r_2)$ with $j_{(m-2)/2, 1} / \lambda > r_2
> r_1 > 0$ such that $B_{r_2} (x)$ is admissible, whereas $r$ either attains each
value in $(r_1, r_2)$ or is equal to $r_i$, $i = 1$ or $2$, then $u$ is metaharmonic
in $D$. Here $j_{(m-2)/2, 1}$ denotes the first positive zero of the Bessel
function $J_{(m-2)/2}$.
\end{theorem}

The assumption $j_{(m-2)/2, 1} / \lambda > r_2$ is imposed to guarantee that the
first factor in the denominator does not vanish.

{\centering \section{Weighted means of pan- and metaharmonic functions} }

In this section, the presentation is restricted to functions of two variables which
facilitates calculations. Let $\partial / \partial n_y$ denote the derivative with
respect to the exterior unit normal $n_y$ at $y \in S_r (x)$. For $w \in C^2 (D)$
and any admissible disc $B_r (x)$ we have
\[ 2 \pi w (x) = \int_{D_r (x)} \nabla^2 w (y) \log |x-y| \, \D y +
\int_{S_r (x)} \left[ w (y) \frac{\partial \log |x-y|}{\partial n_y} -
\frac{\partial w}{\partial n_y} \log |x-y| \right] \D S_y \, .
\]
This identity usually serves as the initial point of the standard proof establishing
mean value properties of harmonic functions. However, it also implies the
representation
\begin{equation}
w (x) = \frac{1}{2 \pi r} \int_{S_r (x)} w (y) \, \D S_y - \frac{1}{2 \pi} \int_{D_r
(x)} \nabla^2 w (y) \log \frac{r}{|x-y|} \, \D y \, , \label{Ev}
\end{equation}
which yields the following weighted version of identity \eqref{MM'}.

\begin{theorem}
Let $D$ be a domain in $\RR^2$. If $u$ is panharmonic in $D$, then
\begin{equation}
a (\mu r) \, u (x) = \frac{1}{\pi r^2} \int_{B_r (x)} u (y) \log \frac{r}{|x-y|} \,
\D y \, , \quad a (t) = \frac{2 \, [I_0 (t) - 1]}{t^2} \, ,
\label{MW'}
\end{equation}
for every admissible disc $B_r (x)$.
\end{theorem}

\begin{proof}
Substituting $u$ into \eqref{Ev} and taking into account identity \eqref{MM}, where
$a^\circ (t) = I_0 (t)$ since $m = 2$, and equation \eqref{Hh}, we obtain
\[ 2 u (x) = 2 a^\circ (\mu r) \, u (x)- \frac{(\mu r)^2}{\pi r^2} \int_{D_r
(x)} u (y) \log \frac{r}{|x-y|} \, \D y \, ,
\]
after multiplying both sides by two. Now \eqref{MW'} follows by rearranging.
\end{proof}

The behaviour of the logarithmic weight in identity \eqref{MW'} is quite simple: it
is a positive function of $y$ within $B_r (x)$, growing from zero attained at $y \in
S_r (x)$ to infinity as $|x-y| \to 0$, and is negative when $y \notin \overline{B_r
(x)}$.

Let us consider some properties of $a$. In view of the definition of $I_0$, we see
that
\[ a (0) = \lim_{t \to +0} a (t) = 1/2 \, ,
\]
and $a (t)$ increases monotonically from this value to infinity similar to
$a^\bullet (t)$. However, the rate of growth as $t \to \infty$ is less for $a (t)$
than for $a^\bullet (t) = 2 \, t^{-1} I_1 (t)$, because $I_0$ and $I_1$ have the
same rate of growth. Moreover, we have that
\[ a^\bullet (t) - a (t) = 2 \left[ t I_1 (t) - I_0 (t) + 1 \right] / t^2
\]
is positive for all $t \in [0, \infty)$, which follows from properties of $I_0$ and
$I_1$. Indeed,
\[ \left[ t I_1 (t) - I_0 (t) + 1 \right]' = I_1 (t) + t I_1' (t)- I_1 (t) = t \,
[ I_2 (t) + I_0 (t) ] / 2 \, ,
\]
which trough straightforward calculations implies that $a^\bullet (t) - a (t)$
increases on $(0, \infty)$ and
\[ \lim_{t \to +0} [ a^\bullet (t) - a (t) ] = \lim_{t \to +0} \frac{\left[ t I_1 (t)
- I_0 (t) + 1 \right]'}{t} = \frac{1}{2} .
\]

An immediate consequence of identity \eqref{MW'} and the fact that $a (t) > 1/2$ for
$t > 0$ is the following.

\begin{corollary}
Let $D$ be a domain in $\RR^2$, and let $u$ be a $\mu$-panharmonic in $D$ for some
$\mu > 0$. If $u \geq 0$ does not vanish identically in $D$, then
\begin{equation}
\frac{1}{2} \, u (x) < \frac{1}{\pi r^2} \int_{B_r (x)} u (y) \log \frac{r}{|x-y|}
\, \D y \label{ineq}
\end{equation}
for every admissible disc $B_r (x)$.
\end{corollary}

Also, Theorem 5.1 allows us to formulate the following analogue of Conjecture~3.1.

\begin{conjecture}
Let $D$ be a bounded domain in $\RR^2$. If $u \in C^0 (D)$ satisfies the identity
\begin{equation*}
\frac{2}{|A (r_1, r_2)|} \int_{A_x (r_1, r_2)} u (y) \log \frac{r}{|x-y|} \, \D y =
a (\mu r) \, M^\circ (S_{r} (x), u) \quad \mbox{with} \ \mu > 0
\end{equation*}
for every $x \in D$ and every $A_x (r_1, r_2)$ with $r_2 > r_1 > 0$ such that
$B_{r_2} (x)$ is admissible, whereas $r$ either attains each value in $(r_1, r_2)$
or is equal to $r_i$, $i = 1$ or $2$, then $u$ is $\mu$-panharmonic in $D$.
\end{conjecture}

Let us turn to real-valued metaharmonic functions of two variables. In this case
the coefficients in identities analogous to \eqref{MM} and \eqref{MM'} are $a^\circ
(t) = J_0 (t)$ and $a^\bullet (t) = 2 \, t^{-1} J_1 (t)$, respectively. It is
interesting to find out how the logarithmic weight changes the last coefficient.
Indeed, minor amendments in the proof of Theorem~5.1 yield the following.

\begin{theorem}
Let $D$ be a domain in $\RR^2$. If $u$ is metaharmonic in $D$, then
\begin{equation}
\tilde a (\lambda r) \, u (x) = \frac{1}{\pi r^2} \int_{B_r (x)} u (y) \log
\frac{r}{|x-y|} \, \D y \, , \quad \tilde a (t) = \frac{2 \, [1 - J_0 (t)]}{t^2} \,
, \label{New}
\end{equation}
for every admissible disc $B_r (x)$.
\end{theorem}

The behaviour of $\tilde a$ is as follows: 
\[ \tilde a (0) = \lim_{t \to +0} \tilde a (t) = 1/2 ,
\]
whereas $\tilde a (t)$ asymptotes zero as $t \to +\infty$ decreasing
nonmonotonically, but remaining positive. The latter property of $\tilde a (t)$
distinguishes it from $2 \, t^{-1} J_1 (t)$ (the coefficient in the area mean value
identity without weight), which has infinitely many zeros.

Also, an analogue of Conjecture~5.1 can be easily formulated for metaharmonic
functions.

\renewcommand{\refname}{
\begin{center}{\Large\bf References}
\end{center}}
\makeatletter
\renewcommand{\@biblabel}[1]{#1.\hfill}
\makeatother

\end{document}